\documentclass[10pt,letterpaper]{article}
\usepackage[utf8]{inputenc}
\usepackage[english]{babel}
\usepackage{amsmath}
\usepackage{amsfonts}
\usepackage{amssymb}
\usepackage{graphicx}

\usepackage{upref,amsthm,amsxtra,exscale}
\usepackage{cite}
\usepackage[colorlinks=true,urlcolor=blue,
citecolor=red,linkcolor=blue,linktocpage,pdfpagelabels,
bookmarksnumbered,bookmarksopen]{hyperref}

\newtheorem{theorem}{Theorem}[section]

\newtheorem{lemma}[theorem]{Lemma}
\numberwithin{equation}{section}

\author{Mónica Clapp\footnote{M. Clapp was partially supported by CONACYT grant 237661 (Mexico) and UNAM-DGAPA-PAPIIT grants IN104315 and IN100718 (Mexico).} \quad and \quad Luis Lopez Rios\footnote{L. Lopez Rios was supported by a postdoctoral fellowship under CONACYT grant 237661 (Mexico).}}
\title{Entire nodal solutions to the pure critical exponent problem for the $p$-Laplacian}
\date{\today}

\begin{document}

\maketitle

\begin{abstract}
We establish the existence of multiple sign-changing solutions to the quasilinear critical problem
$$-\Delta_{p} u=|u|^{p^*-2}u, \qquad u\in D^{1,p}(\mathbb{R}^{N}),$$
for $N\geq4$, where $\Delta_{p}u:=\mathrm{div}(|\nabla u|^{p-2}\nabla u)$ is the $p$-Laplace operator, $1<p<N$ and $p^*:=\frac{Np}{N-p}$ is the critical Sobolev exponent.
\end{abstract} \vspace{6pt}

\section{Introduction}
\label{sec:introduction}
This paper is concerned with the existence of sign-changing solutions to the quasilinear critical problem
\begin{equation}
\label{prob}
-\Delta_{p} u=|u|^{p^*-2}u, \qquad u\in D^{1,p}(\mathbb{R}^{N}),
\end{equation}
for $N\geq4$, where $\Delta_{p}u:=\mathrm{div}(|\nabla u|^{p-2}\nabla u)$ is the $p$-Laplace operator, $1<p<N$ and $p^*:=\frac{Np}{N-p}$ is the critical Sobolev exponent.

It was recently shown in \cite{dmms,s,v} that this problem has a unique positive solution, up to translations and dilations, given by
$$U(x)=a_{N,p} \left( \frac{1}{1+|x|^{\frac{p}{p-1}}} \right)^\frac{N-p}{p},$$
where $a_{N,p}$ is a positive constant. This result extends the one for $p=2$ which was proved in \cite{cgs}. However, as far as we know, no sign-changing solutions to the problem (\ref{prob}) have been found, aside from the semilinear case $p=2$.

For $p=2$ the existence of sign-changing solutions was first established by W. Ding in \cite{di}, who took advantage of the invariance of the problem (\ref{prob}) under Möbius transformations to derive the existence of infinitely many sign-changing solutions. Later, new sign-changing solutions were exhibited by del Pino, Musso, Pacard, and Pistoia in \cite{dmpp}, who used the Lyapunov-Schmidt reduction method to establish the existence of sign-changing clusters of bubbles that solve problem (\ref{prob}) for $p=2$. 

Neither one of these methods applies to the quasilinear case. The $p$-Laplacian is invariant under Euclidean motions and dilations, but it is not invariant under the Kelvin transform, or any suitable version of it, except in the cases $p=2$ and $p=N$; see \cite{l}. So the argument in \cite{di} cannot be extended to other values of $p$. On the other hand, the Lyapunov-Schmidt reduction method used in \cite{dmpp} cannot be applied to the quasilinear case because the linearized operator for the $p$-Laplacian is not well understood for $p\neq2$.

A different type of sign-changing solutions to the problem (\ref{prob}), for $p=2$, was recently found  by the first author in \cite{c}. These solutions were obtained by combining the use of suitable symmetries with concentration arguments. We will show that this approach can be applied to the quasilinear case to prove the following result.

\begin{theorem}
\label{thm:main}
Let $N=4n+m$ with $n\geq1$ and $m\in\{0,1,2,3\}$. Then, for any $1<p<N$, the problem $(\ref{prob})$ has at least $n$ nonradial sign-changing solutions.
\end{theorem}

It is worth noting that every solution to problem (\ref{prob}) belongs to $\mathcal{C}_\mathrm{loc}^{1,\alpha}(\mathbb{R}^N)$ for some $\alpha \in (0,1)$, and satisfies the decay estimates
$$|u(x)| \leq C_0(1+|x|^{\frac{N-p}{p-1}})^{-1} \quad \text{and} \quad |\nabla u(x)| \leq C_0(1+|x|^{\frac{N-1}{p-1}})^{-1},$$
for every $x \in \mathbb{R}^N$. These estimates were recently obtained by Vétois in \cite{v}.

We also mention that positive and sign-changing solutions to the quasilinear equation (\ref{prob}) in some bounded domains have been exhibited in \cite{ct,mp,mss}. Multiplicity of entire solutions to a related quasilinear critical problem, obtained by adding a suitable term to problem (\ref{prob}), was recently established in \cite{bcmp}, although nothing is said about their sign.

The solutions given by Theorem \ref{thm:main} arise as limit profiles of minimizing $\phi$\textit{-equivariant} Palais-Smale sequences for the energy functional associated to the problem
\begin{equation}
\label{eq:general}
-\Delta_{p} u=|u|^{p^*-2}u, \qquad u\in D_0^{1,p}(\mathbb{B}),
\end{equation}
in the unit ball $\mathbb{B}$ in $\mathbb{R}^N$. A $\phi$-equivariant function is a function with a particular type of sign-changing symmetries; the precise definition is given in the following section. We prove a representation theorem for these sequences; see Theorem \ref{thm:profile} below. This result yields an existence alternative: it says that the energy functional has a $\phi$-equivariant minimizer, either in the unit ball, or in a half-space, or in the whole Euclidean space $\mathbb{R}^N$. Moreover, we will prove that the energy of these minimizers is the same in any one of these domains; see Lemma \ref{lem:leastenergy}. So, after trivial extension, this allows us to conclude that the energy functional has a $\phi$-equivariant minimizer in $\mathbb{R}^N$.

If $p=2$ it is well known that the problem \eqref{eq:general} does not have a nontrivial solution, neither in the unit ball, nor in a half-space. But if $p\neq 2$ it is not known whether this is true or not, because the validity of the unique continuation principle is still an open question. So, in principle, there could be solutions to the problem (\ref{prob}) which vanish in some open set.

The multiplicity statement in Theorem \ref{thm:main} is obtained by considering various symmetries which give rise to different solutions.

This paper is organized as follows: in Section \ref{sec:profile} we introduce our symmetric setting and prove a representation theorem for minimizing $\phi$-equivariant Palais-Smale sequences in a bounded symmetric domain. In Section \ref{sec:existence} we prove our main result. Some facts needed for the proof of the representation theorem are proved in the appendix.

\section{The limit profile of a nodal symmetric Palais-Smale sequence}
\label{sec:profile}

As in \cite{c}, we consider the following symmetric setting.

Let $G$ be a closed subgroup of the group $O(N)$ of linear isometries of $\mathbb{R}^{N}$ and let $\phi:G\rightarrow\mathbb{Z}_{2}:=\{1,-1\}$ be a continuous homomorphism of groups. Recall that the $G$-orbit of a point $x\in\mathbb{R}^{N}$ is the set $Gx:=\{gx:g\in G\}.$ 

Hereafter, we will assume that $G$ and $\phi$ have the following properties:

\begin{itemize}
\item[$(\mathbf{S1})$]For each $x \in \mathbb{R}^N$, either $\dim(Gx)>0$ or $Gx=\{x\}$.
\item[$(\mathbf{S2})$]$\phi:G\rightarrow\mathbb{Z}_{2}$ is surjective.
\item[$(\mathbf{S3})$]There exists $\xi\in\mathbb{R}^{N}$ such that $\{g\in G:g\xi=\xi\}\subset\ker\phi$.
\end{itemize}

Let $\Omega$ be a $G$-invariant domain in $\mathbb{R}^N$, i.e., $Gx\subset\Omega$ if $x\in\Omega$. A function $u:\Omega\rightarrow\mathbb{R}$ will be called $\phi$\emph{-equivariant} if
\begin{equation*}
u(gx)=\phi(g)u(x)\qquad \text{for all}\,\, g\in G,\,x\in\Omega.
\end{equation*}
Note that, as $\phi$ is surjective, every nontrivial $\phi$-equivariant function is nonradial and changes sign.

Let $D^{1,p}(\mathbb{R}^N):=\{u\in L^{p^*}(\mathbb{R}^N):\nabla u \in L^p(\mathbb{R}^N,\mathbb{R}^N)\}$ be the Banach space whose norm is given by
\begin{equation*}
\|u\|:=\left( \int_{\mathbb{R}^N}|\nabla u|^p \right)^\frac{1}{p}.
\end{equation*}
As usual, we write $D_0^{1,p}(\Omega)$ for the closure of $\mathcal{C}_c^\infty(\Omega)$ in $D^{1,p}(\mathbb{R}^N)$. We define
\begin{equation*}
D_0^{1,p}(\Omega)^\phi:=\{u\in D_0^{1,p}(\Omega):u\,\text{is $\phi$-equivariant} \}.
\end{equation*}
Property $(\mathbf{S3})$ ensures that this space is infinite dimensional; see \cite{bcm}.

The $\phi$-equivariant solutions to the problem
$$ -\Delta_{p} u=|u|^{p^*-2}u, \qquad u\in D_0^{1,p}(\Omega),$$
are the critical points of the $\mathcal{C}^1$-functional $J:D_0^{1,p}(\Omega)^\phi \to \mathbb{R}$ given by
$$J(u):=\frac{1}{p}\|u\|^p-\frac{1}{p^*}|u|_{p^*}^{p^*},$$
where $|u|_{p^*}^{p^*} := \int_{\Omega}|u|^{p^*}$; see Lemma \ref{lem:symcrit}. The nontrivial ones belong to the set
$$\mathcal{N}^\phi(\Omega):=\{u\in D_0^{1,p}(\Omega)^\phi:u\neq0,\,\|u\|^p=|u|_{p^*}^{p^*}\}.$$
Define
$$c^\phi(\Omega):=\inf_{u\in \mathcal{N}^\phi(\Omega)}J(u).$$

The following facts are well known. We include their proof for the sake of completeness.

\begin{lemma}
\label{lem:nehari}
\begin{enumerate}
	\item[$(a)$]There exists $a_0 >0$ such that $\|u\| \geq a_0$ for every $u\in \mathcal{N}^\phi(\Omega)$.
	\item[$(b)$]$\mathcal{N}^\phi(\Omega)$ is a $\mathcal{C}^1$-Banach submanifold of $D_0^{1,p}(\Omega)^\phi$, and a natural constraint for $J$. 
	\item[$(c)$]Let $\mathcal{T}:=\left\{\sigma \in \mathcal{C}^0\left([0,1],D_0^{1,p}(\Omega)^\phi\right) :\sigma(0)=0,\,\sigma(1)\neq 0, \,J(\sigma(1)) \leq 0\right\}.$ Then,
	$$c^\phi(\Omega)=\inf_{\sigma \in \mathcal{T}} \max_{t \in [0,1]}J(\sigma(u)).$$
\end{enumerate}
\end{lemma}

\begin{proof}
$(a):$ By Sobolev's inequality, there exists $C>0$ such that
$$F(u):= \|u\|^p - |u|_{p^*}^{p^*} \geq \|u\|^p - C \|u\|^{p^*}\qquad\text{for every }\,u \in D_0^{1,p}(\Omega).$$
Hence, there exists $a_0$ such that $F(u)>0$ if $0 < \|u\| < a_0$. This proves $(a)$.

$(b):$ It follows from $(a)$ that $\mathcal{N}^\phi(\Omega)$ is closed in $D_0^{1,p}(\Omega)^{\phi}$. Moreover, as
$$F'(u)u = p\|u\|^p - p^*|u|_{p^*}^{p^*} = (p-p^*)\|u\|^p < 0\quad\text{for every }\,u\in \mathcal{N}^\phi(\Omega),$$
we have that $0$ is a regular value of $F:D_0^{1,p}(\Omega)^{\phi}\smallsetminus\{0\} \to \mathbb{R}$. Hence, $\mathcal{N}^\phi(\Omega)$ is a $\mathcal{C}^1$-Banach submanifold of $D_0^{1,p}(\Omega)^\phi$. This inequality also implies that $\mathcal{N}^\phi(\Omega)$ is a natural constraint for $J$.

$(c):$ For each $u\in\mathcal{N}^\phi(\Omega)$, the function 
$$t\mapsto J(tu)= \left(\frac{t^p}{p} - \frac{t^{p^*}}{p^*}\right)\|u\|^p$$ 
is strictly increasing in $(0,1)$ and strictly decreasing in $(1,\infty)$, and there exists $s_u > 1$ such that $J(s_uu)<0$. So, setting $\sigma_u(t):=ts_uu$ we have that $\sigma_u \in \mathcal{T}$ and $\max_{t\in[0,1]}J(\sigma_u (t)) = J(u)$. Therefore,
$$\inf_{\sigma \in \mathcal{T}} \max_{t \in [0,1]}J(\sigma(u)) \leq \inf_{u\in\mathcal{N}^\phi(\Omega)}\max_{t\in[0,1]}J(\sigma_u (t)) = \inf_{u\in\mathcal{N}^\phi(\Omega)}J(u)=c^\phi(\Omega).$$
To prove the opposite inequality, we define $\kappa:D_0^{1,p}(\Omega)^{\phi}\to\mathbb{R}$ as
\begin{equation*}
\kappa(u):=\left\{
\begin{array}{cl}
\frac{|u|_{p^*}^{p^*}}{\|u\|^p}&\text{if }\,u\neq 0,\\ [4pt]
0&\text{if }\,u=0.
\end{array}
\right.
\end{equation*}
This function is continuous thanks to Sobolev's inequality. Note that $\kappa(u)=1$ iff $u\in \mathcal{N}^\phi(\Omega)$. Moreover, if $J(u) \leq 0$ and $u\neq 0$, then $\kappa(u) \geq \frac{p^*}{p} >1$. So, if $\sigma \in \mathcal{T}$, then $\kappa(\sigma(0))=0$ and $\kappa(\sigma(1))>1$. Hence, there exists $t_0 \in (0,1)$ such that $\sigma(t_0)\in \mathcal{N}^\phi(\Omega)$ and, consequently, $\max_{t \in [0,1]}J(\sigma(t)) \geq J(\sigma(t_0)) \geq c^\phi(\Omega)$. This implies that
$$\inf_{\sigma \in \mathcal{T}} \max_{t \in [0,1]}J(\sigma(u)) \geq c^\phi(\Omega),$$
and finishes the proof of $(c)$.
\end{proof}

\begin{lemma}
\label{lem:willem}
There exists a sequence  $(u_k)$ such that
$$u_k\in D^{1,p}_0(\Omega)^{\phi},\quad J(u_k) \to c^{\phi}(\Omega),\quad \text{and }\quad J'(u_k) \to 0\, \text{ in }\,(D^{1,p}_0(\Omega)^{\phi})'.$$ 
\end{lemma}

\begin{proof}
This follows immediately from statements $(a)$ and $(c)$ of Lemma \ref{lem:nehari}, and \cite[Theorem 2.9]{w}.
\end{proof}

Next, we shall describe the limit profile of these sequences. 

If $X$ is a $G$-invariant subset of $\mathbb{R}^N$, we denote by
$$X^G:=\{x\in X:Gx=\{x\}\}$$
the set of $G$-fixed points in $X$. We start with the following lemmas.

\begin{lemma}
\label{lem:leastenergy}
If $\Omega$ is a $G$-invariant domain in $\mathbb{R}^N$ and $\Omega^G \neq \emptyset$, then
$$c^\phi(\Omega)=c^\phi(\mathbb{R}^N)=:c_{\infty}^\phi.$$
\end{lemma}

\begin{proof}
Clearly, $c_{\infty}^\phi \leq c^\phi(\Omega)$. For the opposite inequality, we fix $x_0\in \Omega^G$ and consider a sequence $(\varphi_k)$ in $\mathcal{N}^\phi(\mathbb{R}^N) \cap \mathcal{C}_c^\infty(\mathbb{R}^N)$ such that $J(\varphi_k) \to c_\infty^\phi$. Since $\varphi_k$ has compact support, we may choose $\varepsilon_k >0$ such that the support of $\tilde{\varphi}_k(x) := \varepsilon_k^{-(N-p)/p} \varphi_k (\varepsilon_k^{-1}(x-x_0))$ is contained in $\Omega$. As $x_0$ is a $G$-fixed point, $\tilde{\varphi}_k$ is $\phi$-equivariant. Thus, we have that $\tilde{\varphi}_k \in \mathcal{N}^\phi(\Omega)$ and, hence,
  \begin{equation*}
    c^\phi(\Omega) \le J(\tilde{\varphi}_k) = J(\varphi_k) \quad \text{for all } k.
  \end{equation*}
Letting $k \to \infty$ we conclude that $c^\phi(\Omega) \leq c_{\infty}^\phi$.
\end{proof}

\begin{lemma}
\label{lem:orbits}
If $G$ satisfies $(\mathbf{S1})$ then, for every pair of sequences $(\varepsilon_k)$ in $(0,\infty)$ and $(x_k)$ in $\mathbb{R}^N$, there exists a sequence $(\xi_k)$ in $\mathbb{R}^N$ such that, after passing to a subsequence,
\begin{equation}
\label{eq:orbits}
\varepsilon_k^{-1}\mathrm{dist}(Gx_k,\xi_k) \leq C_0 \qquad\text{for all }\,k
\end{equation}
and some positive constant $C_0$, and one of the following statements holds true:
\begin{enumerate}
	\item[$(a)$]either $\xi_k \in \Omega^G$,
	\item[$(b)$]or, for each $m \in \mathbb{N}$, there exist $g_1,...,g_m \in G$ such that
	$$\varepsilon_k^{-1} |g_i\xi_k - g_j\xi_k| \to \infty \quad\text{as }\,k \to \infty \qquad\text{if }\,i \neq j.$$
\end{enumerate}
\end{lemma}

\begin{proof}
Write $x_k = z_k + y_k$ with $z_k \in (\mathbb{R}^N)^G$ and $y_k \in \left((\mathbb{R}^N)^G\right)^{\bot}$.

If $(\varepsilon_k^{-1} y_k)$ contains a bounded subsequence, taking such a subsequence and setting $\xi_k:=z_k$ we obtain the statements \eqref{eq:orbits} and $(a)$.

If $(\varepsilon_k^{-1} y_k)$ does not contain a bounded subsequence, passing to a subsequence we have that $\varepsilon_k^{-1} y_k \neq 0$ and
$$\frac{\varepsilon_k^{-1} y_k}{|\varepsilon_k^{-1} y_k|} = \frac{y_k}{|y_k|} \to y \qquad\text{as }\,k \to \infty.$$
Since the $G$-orbit of every point which is not in $(\mathbb{R}^N)^G$ has positive dimension, for each $m \in \mathbb{N}$ there exist $g_1,...,g_m \in G$ such that $g_iy \neq g_jy$ if $i \neq j$. Hence, there exist $k_0 \in \mathbb{N}$ and $\delta >0$ such that 
$$\left|g_i \frac{y_k}{|y_k|} - g_j \frac{y_k}{|y_k|} \right| \geq \delta \qquad\text{for all }\,k\geq k_0\quad\text{if }\,i\neq j.$$
It follows that
$$\varepsilon_k^{-1}|g_ix_k -g_jx_k| \geq \varepsilon_k^{-1}|g_iy_k -g_jy_k| \geq \delta\varepsilon_k^{-1}|y_k| \to \infty.$$
Setting $\xi_k:=x_k$ we obtain the statements \eqref{eq:orbits} and $(b)$.
\end{proof}

\begin{theorem}
\label{thm:profile}
Assume $(\mathbf{S1})-(\mathbf{S3})$. Let $\Omega$ be a $G$-invariant bounded smooth domain in $\mathbb{R}^N$ and $(u_k)$ be a sequence such that
$$u_k\in D^{1,p}_0(\Omega)^{\phi},\quad J(u_k) \to c^{\phi}(\Omega),\quad \text{and }\quad J'(u_k) \to 0\, \text{ in }\,(D^{1,p}_0(\Omega)^{\phi})'.$$ 
Then, up to a subsequence, one of the following two possibilities occurs: 
\begin{enumerate}
\item[\emph{(I)}]either $(u_k)$ converges strongly in $D_0^{1,p}(\Omega)$ to a minimizer of $J$ on $\mathcal{N}^\phi(\Omega)$,
\item[\emph{(II)}]or there exist a sequence of $G$-fixed points $(\xi_k)$ in $\mathbb{R}^N$, a sequence $(\varepsilon_k) \in (0,\infty)$ and a nontrivial solution $W$ to the problem
\begin{equation}
\label{Hprob}
-\Delta_{p} w=|w|^{p^*-2}w, \qquad w\in D_0^{1,p}(\mathbb{H}),
\end{equation}
with the following properties:
	\begin{enumerate}
	\item[\emph{(i)}]$\varepsilon_k \to 0$,\, $\xi_k \to \xi$, \, $\xi\in (\bar{\Omega})^G$,\, and\, $\varepsilon_k^{-1}\mathrm{dist}(\xi_k,\Omega)\to d\in [0,\infty]$.
	\item[\emph{(ii)}]If $d=\infty$, then $\mathbb{H} = \mathbb{R}^N$ and \,$\xi_k \in \Omega$. 
	\item[\emph{(iii)}]If $d\in [0,\infty)$, then $\xi \in \partial \Omega$ and $\mathbb{H} = \{ x \in \mathbb{R}^N: x \cdot \nu > \bar d  \}$, where $\nu$ is the inward pointing unit normal to $\partial \Omega$ at $\xi$ and $\bar d \in \{ d,-d \}$. Moreover, $\mathbb{H}^G\neq\emptyset$ and $\Omega^G\neq\emptyset$.
	\item[\emph{(iv)}]$W\in\mathcal{N}^\phi(\mathbb{H})$ and $J(W)=c_{\infty}^{\phi}$. 
	\item[\emph{(v)}]$\lim\limits_{k\to\infty}\left\| u_k-\varepsilon_k^{-\frac{N-p}{p}} W\left( \frac{x-\xi_k}{\varepsilon_k}  \right) \right\|=0$.
	\end{enumerate}
\end{enumerate}
\end{theorem}

\begin{proof}
As $p>1$ and
\begin{equation}
\label{eq:bounded}
\frac{1}{N}\|u_k\|^p = J(u_k) - \frac{1}{p^*}J'(u_k)u_k \leq C+o(1)\|u_k\|,
\end{equation}
the sequence $(u_k)$ is bounded and, after passing to a subsequence, $u_k\rightharpoonup u$ weakly in $D_0^{1,p}(\Omega)^{\phi}$. Then, $J'(u)=0$; see Lemma \ref{lem:solution}.  We consider two cases:

(I) If $u\neq 0$, then $u\in \mathcal{N}^{\phi}(\Omega)$ and from \eqref{eq:bounded} and our assumptions we obtain
$$c^{\phi}(\Omega) \leq J(u) = \frac{1}{N}\|u\|^p \leq \liminf_{k\to\infty}\frac{1}{N}\|u_k\|^p = c^{\phi}(\Omega)+o(1).$$
Hence, $u_k \to u$ strongly in $D_0^{1,p}(\Omega)^{\phi}$ and $J(u)=c^{\phi}(\Omega)$.

(II) Assume that $u=0$. As
$$\int_{\Omega}|u_k|^{p^*} = N(J(u_k) - \frac{1}{p}J'(u_k)u_k) \to Nc^{\phi}(\Omega),$$
for a fixed $\delta \in (0,\frac{N}{2}c^\phi(\Omega))$ there are bounded sequences $(\varepsilon_k)$ in $(0,\infty)$ and $(x_k)$ in $\mathbb{R}^N$ such that, after passing to a subsequence,
\begin{equation*}
  \delta = \sup_{x\in\mathbb{R}^N}\int_{B_{\epsilon_k}(x)}|u_k|^{p^*}=\int_{B_{\epsilon_k}(x_k)}|u_k|^{p^*},
\end{equation*}
where $B_r(x):=\{z\in \mathbb{R}^N:|z-x|<r\}$. For these sequences we take $(\xi_k)$ as in Lemma \ref{lem:orbits}. Then, $|g_k x_k-\xi_k| \leq C_0\varepsilon_k$ for some $g_k \in G$ and, as $|u_k|$ is $G$-invariant, setting $C_1:=C_0+1$, we have that
\begin{equation}
\label{eq:delta}
  \delta=\int_{B_{\varepsilon_k}(g_k x_k)}|u_k|^{p^*} \leq \int_{B_{C_1\varepsilon_k}(\xi_k)}|u_k|^{p^*}.
\end{equation}
This implies, in particular, that
\begin{equation}
\label{eq:dist}
\mathrm{dist}(\xi_k,\Omega) \leq C_1\varepsilon_k.
\end{equation}

We claim that $\xi_k \in (\mathbb{R}^N)^G$. Otherwise, for each $m\in \mathbb{N}$, Lemma \ref{lem:orbits} would yield $m$ elements $g_1,...,g_m \in G$ such that $B_{C_1\varepsilon_k}(g_i \xi_k) \cap  B_{C_1\varepsilon_k}(g_j \xi_k)=\emptyset$ if $i \neq j$, for $k$ large enough, and from \eqref{eq:delta} we would get that
$$m\delta \leq \sum\limits_{i=1}^m \int_{B_{C_1\varepsilon_k}(g_i\xi_k)}|u_k|^{p^*} \leq \int_{\Omega}|u_k|^{p^*} = Nc^{\phi}(\Omega) + o(1),$$
for every $m\in \mathbb{N}$, which is a contradiction. This proves that $\xi_k \in (\mathbb{R}^N)^G$.

Define $\Omega_k:=\{y\in \mathbb{R}^N: \varepsilon_ky + \xi_k \in \Omega \}$ and, for $y \in \Omega_k$, set
\begin{equation*}
  w_k(y):=\varepsilon_k^\frac{N-p}{p}u_k(\varepsilon_ky + \xi_k).
\end{equation*}
As $u_k$ is $\phi$-equivariant and $\xi_k$ is a $G$-fixed point, $w_k$ is $\phi$-equivariant. Moreover $(w_k)$ is bounded in $D^{1,p}(\mathbb{R}^N)$. Hence, a subsequence satisfies that $w_k \rightharpoonup W$ weakly in $D^{1,p}(\mathbb{R}^N)^{\phi}$, $w_k \to W$ a.e. in $\mathbb{R}^N$ and $w_k \to W$ strongly in $L^{p^*}_{\mathrm{loc}}(\mathbb{R}^N)$. Note that $W$ is $\phi$-equivariant. Choosing $\delta$ sufficiently small and using \eqref{eq:delta}, a standard argument shows that $W\neq 0$; see, e.g., \cite[Section 8.3]{w}.

Passing to a subsequence, we have that $\xi_k \to \xi \in (\mathbb{R}^N)^G$ and $\varepsilon_k \to \varepsilon$. Moreover, $\varepsilon=0$; otherwise, as $u_k \rightharpoonup 0$ weakly in $D_0^{1,p}(\Omega)$, we would have that $W=0$. Furthermore,
\begin{equation*}
  \varepsilon_k^{-1} \mathrm{dist}(\xi_k,\partial\Omega) \to d \in [0,\infty] \quad \text{as } k \to \infty.
\end{equation*}
We consider two cases:
\begin{itemize}
\item[(a)] If $d=\infty$ then, by \eqref{eq:dist}, we have that $\xi_k \in\Omega$. Hence, for every compact subset $X$ of $\mathbb{R}^N$, there exists $k_0$ such that $X\subset \Omega_k$ for all $k\ge k_0$. In this case we set $\mathbb{H}:=\mathbb{R}^N$.
\item[(b)] If $d \in [0,\infty)$ then, as $\varepsilon_k \to 0$, we have that $\xi \in \partial \Omega$. If a subsequence of $(\xi_k)$ is contained in $\bar \Omega$ we set $\bar d:=-d$, otherwise we set $\bar d:=d$. We define 
  \begin{equation*}
   \mathbb{H} :=\{ y\in\mathbb{R}^N: y\cdot\nu > \bar d \},
  \end{equation*}
where $\nu$ is the inward pointing unit normal to $\partial \Omega$ at $\xi$. Since $\xi$ is a $G$-fixed point, so is $\nu$. Thus, $\Omega^G \neq\emptyset$, $\mathbb{H}$ is $G$-invariant and $\mathbb{H}^G \neq \emptyset$. It is easy to see that, if $X$ is compact and $X \subset \mathbb{H}$, there exists $k_0$ such that $X\subset \Omega_k$ for all $k\ge k_0$. Moreover, if  $X$ is compact and $X \subset \mathbb{R}^N \smallsetminus \bar{\mathbb{H}}$, then $X\subset \mathbb{R}^N \smallsetminus\Omega_k$ for $k$ large enough. As $w_k \to W$ a.e. in $\mathbb{R}^N$, this implies, in particular, that $W=0$ a.e. in $\mathbb{R}^N \smallsetminus \mathbb{H}$. So $W \in D_0^{1,p}(\mathbb{H})^{\phi}$.
\end{itemize}

If $\varphi,\psi\in\mathcal{C}_c^{\infty}(\mathbb{H})$, $\varphi$ is $\phi$-equivariant and $\psi$ is $G$-invariant, we define $\varphi_k (x):=\varepsilon_k^{-(N-p)/p}\varphi(\varepsilon_k^{-1}(x - \xi_k))$ and $\psi_k (x):=\varepsilon_k^{-(N-p)/p} [\psi T(w_k-W)](\varepsilon_k^{-1}(x-\xi_k))$, where $T$ is the truncation given by \eqref{eq:truncation}. Then, $\varphi_k$ and $\psi_k$ are $\phi$-equivariant. As $\mathrm{supp}(\varphi)\cup\mathrm{supp}(\psi)\subset \Omega_k$ for $k$ large enough, we have that $\mathrm{supp}(\varphi_k)\subset \Omega$ and $\mathrm{supp}(\psi_k)\subset \Omega$ for $k$ large enough and, since the sequences $(\varphi_k)$ and $(\psi_k)$ are bounded in $D^{1,p}_0(\Omega)^{\phi}$, we get that
\begin{align*}
&\int_{\Omega_k}|\nabla w_k|^{p-2}\nabla w_k \cdot \nabla \varphi - \int_{\Omega_k}|w_k|^{p^*-2}w_k\varphi = J'(w_k)\varphi_k = o(1),\\
&\int_{\Omega_k}|\nabla w_k|^{p-2}\nabla w_k \cdot \nabla [\psi T(w_k-W)] - \int_{\Omega_k}|w_k|^{p^*-2}w_k[\psi T(w_k-W)]\\
 &\qquad= J'(w_k)\psi_k= o(1).
\end{align*}
It follows from Lemma \ref{lem:solution} that $W$ is a nontrivial solution to \eqref{Hprob}. 

From Lemma \ref{lem:leastenergy} we obtain that $c^{\phi}(\Omega)=c^{\phi}(\mathbb{H})=c^{\phi}_{\infty}$. Hence,
$$c^{\phi}_{\infty} \leq \frac{1}{N}\|W\|^p \leq \liminf_{k\to\infty}\frac{1}{N}\|w_k\|^p = \liminf_{k\to\infty}\frac{1}{N}\|u_k\|^p = c^{\phi}_{\infty}.$$
Therefore, $J(W)=c^{\phi}_{\infty}$ and $w_k \to W$ strongly in $D^{1,p}(\mathbb{R}^N)$. After a change of variable,
$$o(1) = \|w_k - W\|^p = \|u_k - \varepsilon_k^{-(N-p)/p}W(\varepsilon_k^{-1}(x - \xi_k))\|^p.$$
This finishes the proof.
\end{proof}

\section{Entire nodal solutions}
\label{sec:existence}

In this section we prove our main result.

\begin{theorem}
\label{thm:existence}
Let $G$ be a closed subgroup of $O(N)$ and $\phi:G\rightarrow\mathbb{Z}_{2}$ be a continuous homomorphism which satisfy $\mathbf{(S1)}-\mathbf{(S3)}$. Then $J$ attains its minimum on $\mathcal{N}^\phi(\mathbb{R}^N)$. Consequently, the problem $(\ref{prob})$ has a nontrivial $\phi$-equivariant solution.
\end{theorem}

\begin{proof}
The unit ball $\Omega:=\{x\in \mathbb{R}^N:|x|<1\}$ is $G$-invariant for every $G$. As $0\in \Omega$, we have that $\Omega^G \neq\emptyset$. So, by Lemma \ref{lem:leastenergy}, $c^\phi(\Omega)=c_{\infty}^\phi$.

By Lemma \ref{lem:willem} there exists a sequence  $(u_k)$ such that
$$u_k\in D^{1,p}_0(\Omega)^{\phi},\quad J(u_k) \to c^{\phi}(\Omega),\quad \text{and }\quad J'(u_k) \to 0\, \text{ in }\,(D^{1,p}_0(\Omega)^{\phi})'.$$ 
Then, Theorem \ref{thm:profile} asserts that there are two possibilities: either there exists $u\in \mathcal{N}^\phi(\Omega)$ with $J(u)=c_{\infty}^\phi$, or there exists $W\in \mathcal{N}^\phi(\mathbb{H})$ with $J(W)=c_{\infty}^\phi$. As $\mathcal{N}^\phi(\Theta)\subset\mathcal{N}^\phi(\mathbb{R}^N)$ for every $G$-invariant domain $\Theta$ in $\mathbb{R}^N$, in either case we conclude that $J$ attains its minimum on $\mathcal{N}^\phi(\mathbb{R}^N)$.
\end{proof}

It is worth noting that in the semilinear case $p=2$ the unique continuation principle excludes the possibility that a solution to the problem (\ref{prob}) vanishes in an open subset of $\mathbb{R}^N$. Therefore, if $\Omega^G\neq\emptyset$, option (II) with $\mathbb{H}=\mathbb{R}^N$ is the only possible option in Theorem \ref{thm:profile}; see \cite[Theorem 2.3]{c}. For other values of $p$ the validity of the unique continuation principle is an open question; see, e.g., \cite{gm}. So one cannot exclude the existence of solutions which vanish in an open subset of $\mathbb{R}^N$.

In order to prove our main result, we need to show that there are groups and homomorphisms with the properties stated in the following lemma. 

\begin{lemma}
\label{lem:examples}
Let $N=4n+m$ with $n\geq1$ and $m\in\{0,1,2,3\}$. Then, for each $j=1,...,n$, there exist a a closed subgroup $G_j$ of $O(N)$ and a continuous homomorphism $\phi_j:G_j\rightarrow\mathbb{Z}_{2}$ with the following properties:
\begin{itemize}
\item[\emph{(a)}]$G_j$ and $\phi_j$ satisfy $\mathbf{(S1)}-\mathbf{(S3)}$.
\item[\emph{(b)}]If $u,v:\mathbb{R}^N \to \mathbb{R}$ are nontrivial functions, $u$ is $\phi_i$-equivariant and $v$ is $\phi_j$-equivariant with $i<j$, then $u\neq v$.
\end{itemize}
\end{lemma}

\begin{proof}
Let $\Gamma$ be the group generated by $\{\mathrm{e}^{\mathrm{i}\theta},\varrho: \theta \in [0,2\pi) \}$, acting on $\mathbb{C}^2$ by
$$\mathrm{e}^{\mathrm{i}\theta}(\zeta_1,\zeta_2):=(\mathrm{e}^{\mathrm{i}\theta}\zeta_1,\mathrm{e}^{\mathrm{i}\theta}\zeta_2), \qquad \varrho(\zeta_1,\zeta_2):=(-\bar{\zeta_2},\bar{\zeta_1}), \qquad \text{for} \quad (\zeta_1,\zeta_2) \in \mathbb{C}^2,$$
and let $\phi:\Gamma \to \mathbb{Z}_2$ be the homomorphism given by $\phi(\mathrm{e}^{\mathrm{i}\theta}):=1$ and $\phi(\varrho):=-1$. Note that the $\Gamma$-orbit of a point $z \in \mathbb{C}^2$ is the union of two circles that lie in orthogonal planes if $z \neq 0$, and it is $\{0\}$ if $z=0$.

Set $\Lambda_j:=O(N-4j)$ if $j=1,...,n-1$, and $\Lambda_n:=\{1\}$. Then the $\Lambda_j$-orbit of a point $y \in \mathbb{R}^{N-4j}$ is an $(N-4j-1)$-dimensional sphere if $j=1,...,n-1$, and it is a single point if $j=n$. 

Define $G_j:=\Gamma^j \times \Lambda_j$, acting coordinatewise on $\mathbb{R}^N \equiv (\mathbb{C}^2)^j \times \mathbb{R}^{N-4j}$, i.e.,
$$(\gamma_1,...,\gamma_j,\eta)(z_1,...,z_j,y):=(\gamma_1 z_1,...,\gamma_j z_j,\eta y),$$
where $\gamma_i \in \Gamma$, $\eta \in \Lambda_j$, $z_i \in \mathbb{C}^2$ and $y \in \mathbb{R}^{N-4j}$, and let $\phi_j:G_j\rightarrow\mathbb{Z}_{2}$ be the homomorphism
$$\phi_j (\gamma_1,...,\gamma_j,\eta):= \phi(\gamma_1) \cdots \phi(\gamma_j).$$
The $G_j$-orbit of $(z_1,...,z_j,y)$ is the product of orbits 
$$G_j(z_1,...,z_j,y) =  \Gamma z_1 \times \cdots \times \Gamma z_j \times \Lambda_j y.$$ 
So, clearly, $G_j$ and $\phi_j$ satisfy $\mathbf{(S1)}-\mathbf{(S3)}$ for each $j=1,...,n$.

Now we prove (b). If $u$ is $\phi_i$-equivariant and $v$ is $\phi_j$-equivariant with $i<j$, and $u(x)=v(x) \neq 0$ for some $x=(z_1,...,z_j,y) \in (\mathbb{C}^2)^j \times \mathbb{R}^{N-4j}$, then, as
$$u(z_1,...,\varrho z_j,y) = u(z_1,...,z_j,y) \quad \text{and} \quad v(z_1,...,\varrho  z_j,y) = -v(z_1,...,z_j,y),$$
we have that $u(z_1,...,\varrho z_j,y) \neq v(z_1,...,\varrho z_j,y)$. This proves that $u\neq v$.
\end{proof} \vspace{4pt}

\begin{proof}[Proof of Theorem \ref{thm:main}]
Let $N=4n+m$ with $n\geq1$ and $m\in\{0,1,2,3\}$. For each $j=1,...,n$, let $G_j$ be the closed subgroup of $O(N)$ and $\phi_j:G_j \to \mathbb{Z}_{2}$ be the continuous homomorphism given by Lemma \ref{lem:examples}. Let $W_j$ be the $\phi_j$-equivariant solution of the problem $(\ref{prob})$ given by Theorem \ref{thm:existence}. Lemma \ref{lem:examples} asserts that the solutions $W_1,...,W_n$ are pairwise distinct.
\end{proof}

Theorem \ref{thm:main} is certainly not optimal. As the proof of Lemma \ref{lem:examples} indicates, there are other possible symmetries which yield further solutions.

\appendix

\section{Appendix}
\label{sec:appendix}

Here we prove Lemma \ref{lem:solution}, which was used in the proof of Theorem \ref{thm:profile}.

Let $\Theta$ be a $G$-invariant domain in $\mathbb{R}^N$. Set
\begin{align*}
\mathcal{C}_c^{\infty}(\Theta)^{\phi}&:=\{\varphi\in\mathcal{C}_c^{\infty}(\Theta):\varphi\text{ is }\phi\text{-equivariant}\},\\
\mathcal{C}_c^{\infty}(\Theta)^G&:=\{\psi\in\mathcal{C}_c^{\infty}(\Theta):\psi\text{ is }G\text{-invariant}\}.
\end{align*}
Recall that $\psi$ is $G$-invariant if it is constant on every $G$-orbit of $\Theta$.

\begin{lemma}
\label{lem:symcrit}
If $u\in D^{1,p}_0(\Theta)^{\phi}$ and $J'(u)\varphi = 0$ for every $\varphi \in \mathcal{C}_c^{\infty}(\Theta)^{\phi}$, then $J'(u)\vartheta = 0$ for every $\vartheta \in \mathcal{C}_c^{\infty}(\Theta)$, i.e., $u$ is a solution to the problem
\begin{equation}
\label{eq:symmprob}
-\Delta_{p} u=|u|^{p^*-2}u, \qquad u\in D_0^{1,p}(\Theta).
\end{equation}
\end{lemma}

\begin{proof}
Let  $\vartheta\in\mathcal{C}_{c}^{\infty}(\Theta)$. Define
$$\varphi(x):=\frac{1}{\mu(G)}\int_{G}\phi(g)\vartheta(gx)\,\mathrm{d}\mu,$$
where $\mu$ is the Haar measure on $G$. Then $\varphi\in\mathcal{C}_c^{\infty}(\Theta)^{\phi}$ and, therefore, $J'(u)\varphi = 0$. Note that, as $u$ is $\phi$-equivariant, $\phi(g)\nabla u(x) = g^{-1}\nabla u(gx)$ for all $g\in G$ and $x\in\Theta$. So, using Fubini's theorem and performing a change of variable, we
get
\begin{align*}
0 &= \int_{\Theta}\left(|\nabla u(x)|^{p-2}\nabla u(x)\cdot\nabla\varphi(x)-|u(x)|^{p^{*}-2}u(x)\varphi(x)\right)\,\mathrm{d}x \\
& = \frac{1}{\mu(G)}\int_{\Theta}\int_{G}\left(|\nabla u(x)|^{p-2}\nabla u(x)\cdot \phi(g)g^{-1}\nabla\vartheta(gx)\right.\\
& \qquad\qquad\qquad\qquad -\left.|u(x)|^{p^{*}-2}u(x)\phi(g)\vartheta(gx)\right)\mathrm{d}\mu\,\mathrm{d}x \\
& = \frac{1}{\mu(G)}\int_{\Theta}\int_{G}\left(|\nabla u(gx)|^{p-2}g^{-1}\nabla u(gx)\cdot g^{-1}\nabla\vartheta(gx)\right.\\
& \qquad\qquad\qquad\qquad -\left.|u(gx)|^{p^{*}-2}u(gx)\vartheta(gx)\right)\mathrm{d}\mu\,\mathrm{d}x\\
& = \frac{1}{\mu(G)}\int_{G}\int_{\Theta}\left(|\nabla u(gx)|^{p-2}\nabla u(gx)\cdot \nabla\vartheta(gx)\right.\\
& \qquad\qquad\qquad\qquad -\left.|u(gx)|^{p^{*}-2}u(gx)\vartheta(gx)\right)\mathrm{d}x\,\mathrm{d}\mu\\
& = \frac{1}{\mu(G)}\int_{G}\mathrm{d}\mu\int_{\Theta}\left(|\nabla u(y)|^{p-2}\nabla u(y)\cdot \nabla\vartheta(y) - |u(y)|^{p^{*}-2}u(y)\vartheta(y)\right)\mathrm{d}y \\
& = \int_{\Theta}\left(|\nabla u(y)|^{p-2}\nabla u(y)\cdot \nabla\vartheta(y) - |u(y)|^{p^{*}-2}u(y)\vartheta(y)\right)\mathrm{d}y,
\end{align*}
i.e., $0=J'(u)\varphi =J'(u)\vartheta$, as claimed.
\end{proof}

Consider the truncation function
\begin{equation}
\label{eq:truncation}
T(t):=
  \left\{
  \begin{aligned}
   & t && \text{if } |t|\leq 1,\\
   & \frac{t}{|t|} && \text{if } |t| \geq 1.
  \end{aligned}
  \right.
\end{equation}
The proof of the following lemma is similar to that of \cite[Lemma 3.5]{ct}. We give the details for the sake of completeness.

\begin{lemma}
\label{lem:ct}
Let $(v_k)$ be a sequence in $D^{1,p}(\mathbb{R}^N)^{\phi}$ and $v\in D_0^{1,p}(\Theta)^{\phi}$ be such that $v_k \rightharpoonup v$ weakly in $D^{1,p}(\mathbb{R}^N)$. Assume that, for every $\psi \in \mathcal{C}_c^{\infty}(\Theta)^G$,
\begin{equation}
\label{eq:ct}
\lim_{k\to\infty}\int_{\Theta}\psi\left(|\nabla v_k|^{p-2}\nabla v_k - |\nabla v|^{p-2}\nabla v\right) \cdot \nabla(T(v_k-v)) = 0.
\end{equation}
Then, after passing to a subsequence, $\nabla v_k \to \nabla v$ a.e. in $\Theta$.
\end{lemma}

\begin{proof}
From the inequalities (4.3) and (4.4) in \cite{l0} we obtain that
\begin{equation}
\label{eq:lindqvist}
(|\eta|^{p-2}\eta - |\xi|^{p-2}\xi)\cdot (\eta - \xi) \geq \left\{
	\begin{aligned}
	& C_0\,|\eta-\xi|^{p} && \text{if }p\geq 2,\\
	& \frac{C_0\,|\eta-\xi|^{2}}{\left(|\xi|^{p}+|\eta|^{p}+1\right)^{2-p}} && \text{if }1<p<2,
	\end{aligned}
	\right.
\end{equation}
for every $\eta,\xi \in \mathbb{R}^N$ and some positive constant $C_0$ which depends only on $p$. 

Set 
$$w_k:=(|\nabla v_k|^{p-2}\nabla v_k - |\nabla v|^{p-2}\nabla v)\cdot(\nabla v_k - \nabla v).$$
By the inequality \eqref{eq:lindqvist}, it suffices to show that, after passing to a subsequence, $w_k \to 0$ a.e. in $\Theta$. 

Note that \eqref{eq:lindqvist} implies that $w_k \geq 0$. Let $\psi \in \mathcal{C}_c^{\infty}(\Theta)^G$ with $\psi \geq 0$ and set $X:=\mathrm{supp}(\psi)$. For each $k,$ we split $X$ into $A_{k}:=\{x\in X:|v_{k}(x)-v(x)| \leq 1\}$ and $B_{k}:=\{x\in X:|v_{k}(x)-v(x)| > 1\}$. After passing to a subsequence, we have that $v_{k}\to v$ in $L^{p}(X)$. Hence, $|B_{k}| \to 0$. Moreover, as $T(v_k-v)=v_k-v$ in $A_k$ and $\nabla (T(v_k-v))=0$ a.e. in $B_k$, we have that
$$\int_{A_{k}}\psi w_{k}=\int_{\Theta}\psi\left(|\nabla v_k|^{p-2}\nabla v_k - |\nabla v|^{p-2}\nabla v\right) \cdot \nabla(T(v_k-v)).$$
Fix $s\in(0,1)$. Since $(v_k)$ is bounded in $D^{1,p}(\mathbb{R}^N)$, using Hölder's inequality and assumption \eqref{eq:ct} we get that
\begin{align*}
0&\leq \int_{\Theta}(\psi w_{k})^{s}=\int_{A_{k}}(\psi w_{k})^{s} + \int_{B_{k}}(\psi w_{k})^{s} \\
&\leq |A_{k}|^{1-s}\left(\int_{A_{k}}\psi w_{k}\right)^{s} + |B_{k}|^{1-s}\left(
\int_{B_{k}}\psi w_{k}\right)^{s} \\
&\leq |X|^{1-s}o(1) + o(1) = o(1).
\end{align*}
So, passing to a subsequence, we have that $\psi w_{k}\to 0$ a.e. in $\Theta$.

Observe that the set $\Theta_m:=\{x\in\Theta: |x|<m,\,\mathrm{dist}(x,\partial \Theta)>\frac{1}{m}\}$ is $G$-invariant for each $m\in\mathbb{N}$. It is easy to construct a $G$-invariant function $\psi_m\in\mathcal{C}_{c}^{\infty}(\Theta)$ such that $\psi_m \geq 0$ and $\psi_m(x)=1$ for every $x\in \Theta_m$. Therefore, passing to a subsequence, $w_k \to 0$ a.e. in $\Theta_m$ for each $m\in\mathbb{N}$. A standard diagonal argument yields a subsequence such that $w_{k}\to 0$ a.e. in $\Theta$, and finishes the proof of the lemma.

\end{proof}

\begin{lemma}
\label{lem:solution}
Let $(v_k)$ be a sequence in $D^{1,p}(\mathbb{R}^N)^{\phi}$ and $v\in D_0^{1,p}(\Theta)^{\phi}$ be such that $v_k \rightharpoonup v$ weakly in $D^{1,p}(\mathbb{R}^N)$. Assume that
$$J'(v_k)\varphi = o(1)\qquad\text{ and }\qquad J'(v_k)[\psi T(v_k-v)] = o(1),$$
for every $\varphi \in \mathcal{C}_c^{\infty}(\Theta)^{\phi}$ and $\psi \in \mathcal{C}_c^{\infty}(\Theta)^G$. Then $v$ is a solution to the problem \eqref{eq:symmprob}.
\end{lemma}

\begin{proof}
First, we claim that $\nabla v_k \to \nabla v$ a.e. in $\Theta$. To prove this claim, we apply Lemma \ref{lem:ct}. Let $\psi \in \mathcal{C}_c^{\infty}(\Theta)^G$. After passing to a subsequence, we have that $v_k \to v$ a.e. in $\Theta$. Then, Egorov's theorem asserts that for every $\delta > 0$ there exists  $A_{\delta}\subset\mathrm{supp}(\psi)$ such that $|A_{\delta}|<\delta$ and $v_k \to v$ uniformly in $\mathrm{supp}(\psi) \smallsetminus A_{\delta}$. So $|v_k(x)-v(x)|\leq 1$ for all $x\in \mathrm{supp}(\psi) \smallsetminus A_{\delta}$ and $k$ large enough. Hence,
\begin{align*}
&\left|\int_{\Theta}\psi|\nabla v|^{p-2}\nabla v \cdot \nabla(T(v_k-v))\right| \\
&\leq \left|\int_{\mathbb{R}^N \smallsetminus A_{\delta}}\psi|\nabla v|^{p-2}\nabla v \cdot \nabla(v_k-v)\right| + \left|\int_{A_{\delta}}\psi|\nabla v|^{p-2}\nabla v \cdot \nabla(T(v_k-v))\right|\\
&\leq o(1) + C\delta,
\end{align*}
because $v_k \rightharpoonup v$ weakly in $D^{1,p}(\mathbb{R}^N)$. Therefore,
\begin{equation}
\label{eq:ct1}
\lim_{k \to \infty}\int_{\Theta}\psi|\nabla v|^{p-2}\nabla v \cdot \nabla(T(v_k-v) = 0.
\end{equation}
On the other hand, as $J'(v_k)[\psi T(v_k-v)] = o(1)$, from Hölder's inequality and the dominated convergence theorem we get that
\begin{align*}
&\left|\int_{\Theta}\psi|\nabla v_k|^{p-2}\nabla v_k \cdot \nabla(T(v_k-v))\right| \\
& \leq \left|\int_{\Theta}|\nabla v_k|^{p-2}\nabla v_k \cdot \nabla(\psi T(v_k-v))\right| + \left|\int_{\Theta}|\nabla v_k|^{p-2}\nabla v_k \cdot T(v_k-v)\nabla\psi \right|\\
& \leq \left|\int_{\Theta}|v_k|^{p^{*}-2} v_k (\psi T(v_k-v))\right| + \left|\int_{\Theta}|\nabla v_k|^{p-2}\nabla v_k \cdot T(v_k-v)\nabla\psi \right| + o(1) \\
& \leq C\left(\int_{\Theta}|\psi T(v_k-v)|^{p^{*}}\right)^{\frac{1}{p^{*}}} + C\left(\int_{\Theta}|T(v_k-v)\nabla\psi|^p \right)^{\frac{1}{p}} + o(1) = o(1).
\end{align*}
Therefore,
\begin{equation}
\label{eq:ct2}
\lim_{k \to \infty}\int_{\Theta}\psi|\nabla v_k|^{p-2}\nabla v_k \cdot \nabla(T(v_k-v) = 0.
\end{equation}
From Lemma \ref{lem:ct}, and identities \eqref{eq:ct1} and \eqref{eq:ct2} we get that $\nabla v_k \to \nabla v$ a.e. in $\Theta$ and, as $v_k \to v$ a.e. in $\Theta$, using again Egorov's theorem we obtain that
\begin{align*}
\lim_{k \to \infty}\int_{\Theta}|\nabla v_k|^{p-2}\nabla v_k \cdot \nabla \varphi &= \int_{\Theta}|\nabla v|^{p-2}\nabla v \cdot \nabla \varphi,\\
\lim_{k \to \infty}\int_{\Theta}|v_k|^{p^*-2} v_k \varphi &= \int_{\Theta}|v|^{p^*-2} v  \varphi,
\end{align*}
for every $\varphi \in \mathcal{C}_c^{\infty}(\Theta)^{\phi}$. Hence,
$$J'(v)\varphi = \lim_{k \to \infty}J'(v_k)\varphi = 0\qquad\text{for every }\,\varphi \in \mathcal{C}_c^{\infty}(\Theta)^{\phi}.$$
So, by Lemma \ref{lem:symcrit}, $v$ is a solution to the problem \eqref{eq:symmprob}, as claimed.
\end{proof}

 \vspace{15pt}

\begin{flushleft}
\textbf{Mónica Clapp}\\
Instituto de Matemáticas\\
Universidad Nacional Autónoma de México\\
Circuito Exterior, Ciudad Universitaria\\
04510 Coyoacán, CDMX\\
Mexico\\
\texttt{monica.clapp@im.unam.mx} \vspace{10pt}

\textbf{Luis Lopez Rios}\\
Instituto de Matemáticas\\
Universidad Nacional Autónoma de México\\
Circuito Exterior, Ciudad Universitaria\\
04510 Coyoacán, CDMX\\
Mexico\\
\texttt{lflopezrios@gmail.com}
\end{flushleft}

\end{document}